\documentclass{amsart}

\usepackage[latin1]{inputenc}
\usepackage{amsmath,amssymb}
\usepackage{pstricks,pst-node}

\newtheorem{thm}{Theorem}[section]

\newtheorem{prop}[thm]{Proposition}
\newtheorem{lem}[thm]{Lemma}
\newtheorem{defn}[thm]{Definition}
\newtheorem{rem}[thm]{Remark}

\def\ROT#1\par{{\red #1\par}}

\long\def\BLAU#1\par{}

\newcommand{\bbA}{\mathbb{A}}

\newcommand{\bbN}{\mathbb{N}}

\newcommand{\bbQ}{\mathbb{Q}}

\newcommand{\bbZ}{\mathbb{Z}}

\newcommand{\calA}{\mathcal{A}}

\newcommand{\Chut}{\hat{C}}

\newcommand{\Fr}{\operatorname{Dec}}
\newcommand{\Ar}{\operatorname{Ar}}
\newcommand{\Aut}{\operatorname{Aut}}

\newcommand{\Gal}{\operatorname{Gal}}
\newcommand{\gal}{\operatorname{\mathit{Gal}}}

\newcommand{\id}{\operatorname{id}}  
\newcommand{\im}{\operatorname{im}}
\newcommand{\Ind}{\operatorname{Ind}}

\newcommand{\KO}{\operatorname{K_0}}
\newcommand{\Norm}{\operatorname{N}}

\newcommand{\Res}{\operatorname{Res}}
\newcommand{\spec}{\operatorname{spec}}
\newcommand{\Th}{\operatorname{Th}}

\newcommand{\Zhut}{\hat{\bbZ}}

\newcommand{\acl}[1]{\tilde{#1}}
\newcommand{\aclK}{\acl{K}}

\newcommand{\absGal}[1]{\Gal(\acl{#1}/#1)}
\newcommand{\GalK}{\absGal{K}}

\newcommand{\Psub}{\operatorname{Psub}}

\newcommand{\Ppart}{\operatorname{Ppart}}

\newcommand{\alphareg}{\alpha_{\mathrm{reg}}}

\newcommand{\Kvar}[1][k]{\KO(\mathrm{Var}_{#1})}

\newcommand{\Kmot}[1][k]{\KO(\mathrm{Mot}_{#1})}
\newcommand{\KmotQ}[1][k]{\KO(\mathrm{Mot}_{#1})_{\bbQ}}
\newcommand{\KT}[1][\gal,k]{\KO(T_{#1})}

\newcommand{\mmu}{\chi_c}   
\newcommand{\zmmu}{\hat{\chi}_{c}}

\newcommand{\calAfine}{\tilde{\calA}}

\newcommand{\ohne}{\smallsetminus}

\newcommand{\surject}{\twoheadrightarrow}
\newcommand{\inject}{\hookrightarrow}

\newcommand{\acts}{\looparrowright}
\newcommand{\dcup}{\mathrel{\dot{\cup}}}

\newcommand{\me}[1]{\{#1\}}
\newcommand{\gen}[1]{\langle#1\rangle}

\newcommand{\auf}[1]{\mathbin{|_{#1}}}

\newcommand{\opp}{^{\mathrm{opp}}}

\newcommand{\overto}[1]{\mathrel{%
  \overset{#1}{%
    {\text{\footnotesize\vphantom{x}}}\smash{\surject}%
  }%
}}
\newcommand{\VGW}[1][]{V#1 \overto{G#1}W#1}

\newcommand{\qqquad}{\qquad\quad}

\newcommand{\pu}{\,.}
\newcommand{\ko}{\,,}

\psset{%
  labelsep=0.5ex,%
  nodesep=1mm,%
  linewidth=0.11mm,%
  arrowsize=1.2mm 1,
  arrowlength=0.5,
  arrowinset=0.8
}

\makeatletter

\def\psas@i{newpath -1.5 1.5 1.5 180 360 arc stroke 0 1.5 moveto}
\def\psas@j{newpath 1.5 1.5 1.5 180 360 arc stroke 0 1.5 moveto}

\def\tx@Arrow{CLW mul add dup 2 div /w ED mul dup /h ED mul /a ED
{ 0 h T 1 -1 scale } if
 0    w mul 0.3  h mul moveto
-0.35 w mul 0.8  h mul  -0.9  w mul 1.02 h mul -1 w mul 1.07 h mul curveto
-1    w mul 0.9  h mul L
-0.9  w mul 0.85 h mul  -0.3 w mul  0.6  h mul  0 w mul 0    h mul curveto
 0.3  w mul 0.6  h mul   0.9 w mul  0.85 h mul  1 w mul 0.9  h mul curveto
 1    w mul 1.07 h mul L
 0.9  w mul 1.02 h mul   0.35 w mul 0.8  h mul  0 w mul 0.3  h mul curveto
gsave fill grestore }

\makeatother

\title{Motives for perfect PAC fields with pro-cyclic Galois group}

\author{Immanuel Halupczok}

\address{DMA\\
Ecole Normale Supérieure\\
45, rue d'Ulm\\
75230 Paris Cedex 05 -- France
}

\email{math@karimmi.de}

\thanks{The author was supported by the Agence National de la Recherche
(contract ANR-06-BLAN-0183-01).}



\begin{document}

\begin{abstract}
Denef and Loeser defined a map from the Grothendieck ring of sets definable
in pseudo-finite fields to the Grothendieck ring of Chow motives,
thus enabling to apply any cohomological invariant to these sets.
We generalize this to 
perfect, pseudo algebraically closed fields with pro-cyclic Galois group.

In addition, we define some maps between different Grothendieck rings
of definable sets which provide additional information,
not contained in the associated motive.
In particular we infer that the map of Denef-Loeser is not injective.
\end{abstract}

\maketitle

\tableofcontents

\section{Preliminaries}

\subsection{Introduction}

To understand definable sets of a theory,
it is helpful to have invariants with nice properties.
For a fixed pseudo-finite field $K$, there are two well-known
invariants of definable sets: the dimension (see \cite{CH:dim}),
and the measure (see \cite{CDM}).

In a slightly different setting,
Denef and Loeser constructed a much stronger invariant:
they do not fix a pseudo-finite field; instead they consider
definable sets in the theory of all pseudo-finite fields of characteristic
zero. To each such set $X$ they associate an element $\mmu(X)$ of the
Grothendieck ring of Chow motives
(see \cite{DL}, \cite{DL2}). In particular, this implies that all the
usual cohomological invariants (like Euler characteristic, Hodge polynomial)
are now applicable to arbitrary definable sets.


The dimension defined in \cite{CH:dim} exists for a much larger class
of fields and in \cite{Hr:meas},
Hrushovski asked whether one can also generalize the measure. This
question has been answered in \cite{i:dm}: 
it is indeed possible to define a measure for any perfect,
pseudo algebraically closed (PAC)
field with pro-cyclic Galois group.
A natural question is now: can the work of Denef-Loeser
also be generalized to this setting?
More precisely, fix a torsion-free pro-cyclic group $\gal$
and consider the theory of perfect PAC fields with absolute Galois group
$\gal$. Then to any definable set $X$ in that theory
we would like to associate a virtual motive $\mmu(X)$.
The first goal of this article is to do this (Theorem~\ref{thm:main}).

One reason this result seems interesting to me is
the following: the map $\mmu$ exists for pseudo-finite fields
(by Denef-Loeser) and for algebraically closed fields
(by quantifier elimination).
The case of general pro-cyclic
Galois groups is a common generalization of both and thus a kind of
interpolation.

Comparing those maps $\mmu$ for different Galois groups, one gets the feeling
that they are closely related. Indeed, given an inclusion of Galois groups
$\gal_2 \subset \gal_1$, we will prove (in Theorem~\ref{thm:reduce})
the existence of a map $\theta$
from the definable sets for $\gal_2$ to the definable sets for $\gal_1$
which is compatible with the different maps $\mmu$.

These maps $\theta$ turn out to be interesting
in themselves.
An open question was whether the map $\mmu$ is injective for
pseudo-finite fields. We will show (Proposition~\ref{prop:example})
that it is not,
by giving an example of two definable sets with the same image under
$\mmu$ but with different images under one of those maps $\theta$.
This also means that at least in this case, the maps $\theta$ can
be used to get information which one loses by applying $\mmu$.

We have one more result. In \cite{DL2}, the map $\mmu$
is defined by enumerating certain properties and then
existence and uniqueness of such a map is proven.
We are able to weaken the conditions needed for uniqueness in the
case of pseudo-finite fields.
Unfortunately however, we do not get any sensible uniqueness conditions
for other pro-cyclic Galois groups.

\subsection{The results in detail}
\label{subsect:results}

Let us fix some notation once and for all.

By a ``group homomorphism'' we will always mean a continuous group
homomorphism if there are pro-finite groups involved.

We fix a field of parameters $k$ and a group $\gal$
which will serve as Galois group.
Sometimes, we will require $k$ to be of characteristic zero.
$\gal$ will always be a pro-cyclic group such that there do exist perfect
PAC fields having $\gal$ as absolute Galois group. This is the case
if and only if $\gal$ is torsion-free, or equivalently,
if it is of the form $\prod_{p \in P} \bbZ_p$, where $P$ is any set of
primes.

\BLAU UG $\Rightarrow$ Gal-Grp ist klar. Gal-Grp $\Rightarrow$ UG:
Prozyklisch $\Rightarrow$ ist Quotient von $\Zhut$. Quotienten kann man
auf jedem Faktor $\bbZ_p$ einzeln betrachten. Entweder man hat im Quotient
ganz $\bbZ_p$ oder nix oder $\bbZ/p^n\bbZ$. Wenn ein $\bbZ/p^n\bbZ$ vorkommt,
ist das eine Untergruppe, also entspricht das einem Oberkörper mit
Galois-Gruppe $\bbZ/p^n\bbZ$. Oder anderes Argument: $\bbZ/p^n\bbZ$ ist
nicht projektiv (im Gegensatz zu Galois-Gruppen von PAC-Körpern).

The theory we will be working in will be the theory of perfect PAC
fields with absolute Galois group $\gal$ which contain
$k$. We will denote this theory by $T_{\gal,k}$.
Models of $T_{\gal,k}$ will be denoted by $K$; the algebraic closure of
a field $K$ will be denoted by $\aclK$. By ``definable'' we
always mean 0-definable. (But $k$ is part of the language.)

By ``variety'', we mean a separated, reduced scheme of finite type.
If not stated otherwise, all our varieties will be over
$k$.\footnote{We will try to limit our notation such that
readers not so familiar with the language of schemes can use
a more naive definition of varieties. For those readers: our
varieties are not supposed to be irreducible.}

\BLAU
Zu Def. von Var: [BRNA] macht das ohne reduced, aber das aendert nix am
Grothendieckring.

We will use the notion ``definable set'' even when there is no model around:
by a ``definable set (in $T_{\gal,k}$)'', we mean a formula
up to equivalence modulo $T_{\gal,k}$.
In addition, we will permit ourselves to
speak about ``definable subsets of (arbitrary) varieties''.
For affine embedded varieties,
it is clear what this should mean. In general, any definable decomposition
of a variety $V$ into affine embedded ones yields the same notion of
definable subsets of $V$ (cf.\ ``definable sub-assignments'' in
\cite{DL}).


We will use the usual definitions of the following Grothendieck rings
(see e.g.\ \cite{DL} or \cite{DL2}): the Grothendieck ring of varieties
$\Kvar$, the Grothendieck ring of (Chow) motives $\Kmot$
and the Grothendieck ring $\KT$ of the theory $T_{\gal,k}$.
Moreover, we will often need to tensor the Grothendieck ring of motives
with $\bbQ$; we denote this by $\KmotQ := \Kmot \otimes_{\bbZ}\bbQ$.


Now let us state the generalization of the theorem of Denef-Loeser.
For the definition of ``Galois cover'' and ``$X(\VGW, \me{1})$'',
see Section~\ref{sect:gal}.

\begin{thm}\label{thm:main}
Suppose
$\gal = \prod_{p \in P} \bbZ_p$ (where $P$ is any set of primes)
is a torsion-free pro-cyclic group and $k$ is a field of characteristic
zero. Then there exists a (canonical) ring homomorphism
$\mmu\colon\KT \to \KmotQ$ extending the usual homomorphism
$\mmu\colon\Kvar \to \Kmot$ with the following property:
if $\VGW$ is a Galois cover such that all prime factors of
$|G|$ lie in $P$, then
\[
(*)\qquad \mmu(X(\VGW, \me{1})) = \frac{1}{|G|}\mmu(V)\pu
\]
If $\gal = \Zhut$, then a ring homomorphism with these properties is
unique.
\end{thm}

As already mentioned, our condition $(*)$ needed for uniqueness in the
pseudo-finite case is weaker than the one of
Denef-Loeser (Theorem~6.4.1 of \cite{DL2}).

If $\gal \ne \Zhut$, we can not prove that condition
$(*)$ is strong enough to define $\mmu$ uniquely,
and we do not have any good replacement for $(*)$.
Nevertheless, we will sometimes speak of \emph{the} map
$\mmu\colon\KT \to \KmotQ$ and mean the one defined
in Section~\ref{subsect:proof} (after Lemma~\ref{lem:mmu}).

\BLAU Kriegt man Eindeutigkeit im folgenden Sinne? $\mmu$ ist die
einzige Abbildung, so dass alle anderen Abbildungen $\KT \to \KmotQ$
über $\mmu$ faktorisieren.

The map $\mmu$ does not really depend on the base field $k$: if we have a second
field $k'$ containing $k$, then there are canonical ring homomorphisms
$\KT \to \KT[\gal,k']$ and $\KmotQ \to \KmotQ[k']$, which we will both denote
by $\otimes_k k'$. The map $\mmu$ is compatible with
these homomorphisms:

\begin{prop}\label{prop:change-k}
In the setting just described we have, for any definable set $X$
of $T_{\gal,k}$, $\mmu(X \otimes_k k') = \mmu(X)\otimes_k k'$.
\end{prop}

We will not write down the proof of this, as it is exactly the same
as in the pseudo-finite case; see \cite{DL}, the paragraph before
Lemma~3.4.1, or \cite{Nic:mot}, Proposition~8.9.

\BLAU
The claim of the proposition is
$\mmu(X \otimes_k k') = \mmu(X)\otimes_k k'$ for any definable set
$X$ of $T_{\gal,k}$. It is enough to check the statement for generators
of $\KT$, i.e.\ we may suppose that $X$ is of the form
$X = X(\VGW, C)$ where $(\VGW, C)$ is a colored Galois cover defined over $k$.

\BLAU
By tensoring with $k'$ and choosing one irreducible component $V'$ of
$V \otimes_k k'$, we get a Galois cover $\VGW[']$ over $k'$, where
$G'$ is a subgroup of $G$. One easily verifies
$X\otimes_k k' = X(\VGW['], C')$ where
$C' := \me{Q \in  C \mid Q \subset G'}$.

\BLAU Das sieht so aus als waeren mit $k'$ ploetzlich weniger Mengen
defbar. Das ist in der Tat dann der Fall, wenn $k'$ fuer dieses spezielle
Galois-cover erzwingt, dass $V$ nicht irreduzibel ist.
Anders ausgedrueckt: Wir wollen $(X\otimes_k k')(K) = X(K)$, aber
nur fuer solche Modelle $K$, die $k'$ enthalten.

\BLAU
Using Lemma~\ref{lem:mmu:prop} we get
\[
\begin{aligned}
\mmu(X(\VGW['], C')) &= \mmu(G'\acts V', \alpha_{C'})
= \mmu(G'\acts V', \Res_{G'}^G\alpha_{C})\\
&= \mmu(G\acts V \otimes_k k', \alpha_{C})
= \mmu(G\acts V, \alpha_{C}) \otimes_k k'\\
&= \mmu(X(\VGW, C)) \otimes_k k'
\pu
\end{aligned}
\]
\BLAU Zum 2. $=$: $\alpha_{C'}(g') = 1 \iff \Ppart(\gen{g'}) \in C'
\iff \Ppart(\gen{g'}) \in C\dots$; letzteres da $\gen{g'} \subset G'$.

\BLAU Das vorletzte $=$ muss man wohl aus dem Beweis in [BRNA] ablesen.
Scheint aber kein Problem zu sein.


The next theorem is the one putting the Grothendieck rings of theories
corresponding to different Galois groups into relation.

\begin{thm}\label{thm:reduce}
Suppose $\gal_1$ and $\gal_2$ are two torsion-free pro-cyclic groups,
$\iota\colon \gal_2 \inject \gal_1$ is an injective map, and 
$k$ is any field (not necessarily of characteristic zero).
Denote the theories $T_{\gal_i,k}$ by $T_i$ for $i=1,2$.
Then the following defines a
ring homomorphism $\theta_\iota\colon \KT[2] \to \KT[1]$:
Suppose $K_1$ is a model of $T_1$. 
Then the fixed field $K_2 := \acl{K}_1^{\iota(\gal_2)}$
is a model of $T_2$ containing $K_1$. For any $X_2 \subset \bbA^n$
definable in $T_2$, we define
$\theta_\iota(X_2)(K_1) := X_2(K_2) \cap K_1^n$.
\end{thm}

\BLAU Es scheint zu gelten: Wenn der Kokern von $\iota$ torsionsfrei
ist, dann ist $\theta_\iota$ injektiv.

Using this theorem, one can reduce the 
existence of $\mmu$ for arbitrary torsion-free pro-cyclic groups $\gal$
to the case $\gal = \Zhut$ (which has been treated by Denef-Loeser):
apply Theorem~\ref{thm:reduce} to $\iota\colon \gal \inject \Zhut$,
where $\iota$ maps $\gal$ to the appropriate factor
$\prod_{p\in P}\bbZ_p$ of $\Zhut$ (such that $\Zhut/\!\gal$ is torsion-free).
Then define $\mmu$ as the composition $\zmmu\circ\theta_\iota$,
where $\zmmu\colon \KT[\Zhut,k] \to  \KmotQ$ is the known map in the
pseudo-finite case.
Verification of the properties of $\mmu$ is not very
difficult using the
explicit computations done in the proof of Theorem~\ref{thm:reduce}.

So in principle, we are done with the existence part of Theorem~\ref{thm:main}
(provided we can prove Theorem~\ref{thm:reduce}).
On the other hand, one has
the feeling that it should also be possible to construct
the map $\mmu$ directly for any group $\gal$. We will do this
in Section~\ref{subsect:proof}, but as our construction closely follows
the construction in \cite{Nic:mot}, we will
go into details only in places where our generalization requires
some modifications.


Another interesting application of Theorem~\ref{thm:reduce}
is the case $\gal_1 = \gal_2 = \gal$,
but with a non-trivial injection $\iota\colon \gal \inject \gal$.
One thus gets endomorphisms of the ring $\KT$,
which might reveal a lot of information about its structure. Indeed using
such endomorphisms we will construct a
whole family of pairs of definable sets $X_1$ and $X_2$
such that $\mmu(X_1) = \mmu(X_2)$ but $\mmu(\theta(X_1)) \ne \mmu(\theta(X_2))$,
thereby proving:
\begin{prop}\label{prop:example}
Let $k$ be a field of characteristic zero and let $\gal$ be a
non-trivial torsion-free pro-cyclic group.
Then the map $\mmu\colon \KT \to \KmotQ$ is not injective.
\end{prop}

The remainder of the article is organized as follows.
In Section~\ref{sect:gal}, we state the main tool to
get hold of arbitrary definable sets, namely
quantifier elimination to Galois formulas. Before that, we introduce
the necessary notation: Galois covers, a generalized Artin symbol
and Galois stratifications.
In Section~\ref{sect:proof}, we prove Theorem~\ref{thm:main}.
Section~\ref{sect:reduce} is devoted to the maps $\theta_{\iota}$:
we prove Theorem~\ref{thm:reduce}
and Proposition~\ref{prop:example},
and moreover, we check that
the maps $\mmu$ of Theorem~\ref{thm:main} for different Galois groups
are compatible with suitable maps $\theta_\iota$
(Proposition~\ref{prop:coincide}).
Finally Section~\ref{sect:open} lists some open problems.

\section{Galois stratifications and quantifier elimination}
\label{sect:gal}

A standard technique to get hold of definable sets of
perfect PAC fields with not-too-large
Galois group is the quantifier elimination to Galois formulas.
In this section, we define the necessary objects and then, 
in Section~\ref{subsect:qe}, state this quantifier
elimination result in the version of
Fried-Jarden \cite{FJ}.



\subsection{Galois covers}

\begin{defn}
\begin{enumerate}
\item
A \emph{Galois cover} consists of two
integral and normal varieties $V$ and $W$ (over some fixed field $k$)
and a finite \'etale map $f\colon V \to W$
such that for $G := \Aut_W(V)\opp$,
we have canonically $W \cong V/G$ (where $G$ acts from the right on $V$).
We denote a Galois cover by $f\colon \VGW$ and call $G$ the \emph{group}
of that cover. The action of $G$ on $V$ will be denoted by $v.g$ (for $v \in V$,
$g \in G$).
\item
We say that a Galois cover $f'\colon V'\overto{G'}W$ is a \emph{refinement} of
$f\colon \VGW$, if there is a finite \'{e}tale map
$g\colon V' \surject V$ such that $f' = f \circ g$.
\item
If $W''$ is a locally closed subset of $W$ and $V''$ is a connected component
of $f^{-1}(W'')$, then we call
$\VGW['']$ the \emph{restriction} of $\VGW$ to $W''$,
where $G'' := \Aut_{W''}(V'')\opp$.
\end{enumerate}
\end{defn}

\begin{rem}
\begin{enumerate}
\item
If $f'\colon V'\overto{G'}W$ is a refinement of $f\colon \VGW$,
then we have a canonical surjection $\pi\colon G' \surject G$.
\item
If $\VGW['']$ is a restriction of $\VGW$, then we have a canonical
injection $G'' \inject G$.
Different choices of the connected component of $f^{-1}(W'')$
yield isomorphic restricted Galois covers.
\end{enumerate}
\end{rem}

\BLAU $G''$ ist UG von $G$: Beide operieren einfach transitiv auf einer
Faser ueber einem Punkt in $W''$...

\subsection{Artin symbols and colorings}

Using a Galois cover $\VGW$, we would like to decompose $W$ into
subsets according to the Artin symbol of the elements. However, the usual
definition of Artin symbol needs a canonical generator of the Galois group
$\gal$ (usually the Frobenius of a finite field); the Artin symbol is then the
image of the generator under a certain map $\rho\colon \gal \to G$
(which is unique only up to conjugation by $G$).
If one does not have such a
canonical generator, then one still can consider the image of $\rho$.
This is what one uses as Artin symbol in our case (see \cite{FJ}).

\begin{defn}[and Lemma]
Suppose $f\colon \VGW$ is a Galois cover over $k$ and $K$ is a field
containing $k$.
\begin{enumerate}
\item
Suppose $v \in V(\aclK)$ such that $f(v) \in W(K)$.
Then there is a unique group
homomorphism $\rho\colon\GalK \to G$ satisfying
$\sigma(v) = v.\rho(\sigma)$ for any $\sigma \in \GalK$.
The \emph{decomposition group} $\Fr(v) := \im \rho \subset G$ of $v$
is the image of that homomorphism.
\item
For $w \in W(K)$, let the \emph{Artin Symbol} $\Ar(w)$ of $w$
be the set $\me{\Fr(v) \mid v \in V(\aclK), f(v) = w}$ of decomposition
groups of all preimages of $w$.
\end{enumerate}
\end{defn}

\BLAU
Die ``richtige'' Art zu schreiben, dass $v$ ein Punkt mit Bild in
$W(K)$ ist, ist: $w\colon \spec K \to W$ ist ein Punkt, und
$v$ ist eine Zshgskomp. von $K \times_W V$. Vgl. \cite{Nic:mot}.

$\Ar(w)$ consists exactly of one conjugacy class of
subgroups of $G$,
and these subgroups are isomorphic to a quotient of the absolute
Galois group $\GalK$ of the field.

If $K$ is a model of our theory $T$, then
the quotients of $\GalK = \gal$ are just the cyclic groups $Q$ such 
that all prime factors of $|Q|$ lie in $P$
(where $P$ is the set of primes such that $\gal = \prod_{p\in P}\bbZ_p$).
We introduce some notation for this:

\begin{defn}\label{defn:Ppart}Given a finite group $G$, we will call those subgroups of $G$ which
are isomorphic to a quotient of $\gal$ the \emph{permitted subgroups}.
We denote the set of all permitted subgroups of $G$ by
$\Psub(G)$. If $Q$ is a finite cyclic group, then we denote by $\Ppart(Q)$
the ``permitted part of $Q$'', i.e.\ the biggest permitted subgroup of $Q$.
\end{defn}

\BLAU
Diese Def. vom Artin-Symbol ist vermutlich nicht die Richtige: Will:
Menge von Homomorphismen $\gal \to G$ modulo $G$-Konj und modulo Automorphismen
von $\gal$. Statt modulo Automorphismen von $\gal$ haben wir aber hier
modulo Automorphismen vom Bild vom Homomorphismus; das koennten mehr sein.
Sieht ein bisschen so aus, als koennte der Fehler der Grund dafuer sein,
dass [FJ] die Bedingung ``embedding property'' an $\gal$ stellt.

The interest
of $\Ppart(Q)$ is the following. We will sometimes identify
$\gal = \prod_{p\in P}\bbZ_p $ with the
corresponding factor of $\Zhut$ and consider homomorphisms
$\rho\colon \Zhut \to G$. Then the image of $\gal$ in $G$
is just $\rho(\gal) = \Ppart(\im\rho)$.

Given a Galois cover $\VGW$, we now define subsets of $W$ using the Artin
symbol:

\begin{defn}
\begin{enumerate}
\item
A \emph{coloring} of a Galois cover $\VGW$ is a subset $C$ of the
permitted subgroups of
$G$ which is closed under conjugation. A Galois cover together
with a coloring is called a \emph{colored Galois cover}.
\item
Given
a colored Galois cover $(\VGW, C)$ and a model $K\models T$,
we define the set
$X(\VGW, C)(K) := \me{w \in W(K) \mid \Ar(w) \subset C}$.
\end{enumerate}
\end{defn}

Note that $X(\VGW, C)$ is definable, i.e.\ there is a formula $\phi$
such that for any model $K\models T$ we have $\phi(K) = X(\VGW, C)(K)$.

\BLAU
Achtung: Kann \emph{nicht} (uniform fuer alle $K$) die Erweiterung vom
Grad $n$ interpretieren. (Keine Gruppoid-imag-elim.)
Kann aber uniform sagen: ``Fuer alle Erweiterungen
vom Grad $n$ gilt: \dots''

\begin{rem}\label{rem:refres}
\begin{enumerate}
\item
If $(\VGW, C)$ is a colored Galois cover and $V'\overto{G'}W$
is a refinement with canonical map $\pi\colon G' \surject G$,
then we can also refine the coloring: by setting
$C' := \me{Q \in \Psub(G')\mid \pi(Q) \in C}$, 
we get $X(V'\overto{G'}W, C') = X(\VGW, C)$.
\item
Similarly if $\VGW['']$ is a restriction of $f\colon\VGW$:
in that case,
set $C'' := \me{Q \in C \mid Q \subset G''}$.
Then we get
$X(\VGW[''], C'') = X(\VGW, C) \cap W''$.
\end{enumerate}
\end{rem}

\subsection{Galois stratifications}

\begin{defn}
A \emph{Galois stratification} $\calA$ of a variety $W$ is a finite family 
$(f_i\colon V_i \overto{G_i} W_i, C_i)_{i\in I}$ of colored Galois covers
where the $W_i$ form a partition of $W$ into locally closed sub-varieties.
We shall say that $\calA$ \emph{defines} the following subset $\calA(K) \subset
W(K)$, where $K\models T$ is a model:
\[
\calA(K) := \bigcup_{i\in I}X(\VGW[_i], C_i)(K)
\]
\end{defn}
The data of a Galois stratification denoted by $\calA$ will always
be denoted by $V_i$, $W_i$, $G_i$, $C_i$, and analogously with primes for
$\calA'$, $\calA''$, etc. This will not always be explicitely mentioned.

\begin{defn}
Suppose $\calA$ and $\calA'$ are two Galois stratifications.
We say that $\calA'$ is a \emph{refinement} of $\calA$, if:
\begin{itemize}
\item
Each $W_i$ is a union $\bigcup_{j\in J_{i}} W'_j$ for some $J_i\subset I'$.
\item
For each $i \in I$ and each $j \in J_i$, the Galois cover
$\VGW['_j]$ is a refinement of the restriction of the Galois cover $\VGW[_i]$
to the set $W'_j$.
\item
$C'_j$ is constructed out of $C_{i}$ as described in Remark~\ref{rem:refres},
such that $X(\VGW['_j], C'_{j}) = X(\VGW[_{i}], C_{i}) \cap W'_{j}$.
\end{itemize}
\end{defn}

By the third condition, $\calA$ and $\calA'$ define the same set.

One reason for Galois stratifications being handy to
use is the following general lemma:

\begin{lem}\label{lem:refine}
If $\calA$ and $\calA'$ are two Galois stratifications, then there exist
refinements $\calAfine$ and $\calAfine'$ of $\calA$ resp.\ $\calA'$
which differ only in the colorings.
\end{lem}

\subsection{Quantifier elimination to Galois stratifications}
\label{subsect:qe}

We now state the version of quantifier elimination which we will
use. It is given in \cite{FJ}, Proposition~30.5.2.
Applied to our situation, that proposition reads:
\begin{lem}\label{lem:qe}Suppose $\gal$ is a torsion-free pro-cyclic group
and $k$ is any field. Then
each definable set $X$ of $T_{\gal,k}$ is already definable by a Galois
stratification $\calA$ (over $k$),
i.e.\ for any $K\models T_{\gal,k}$, we have $X(K) = \calA(K)$.
\end{lem}
Note that Proposition~30.5.2 of \cite{FJ} requires that $K$
is what Fried-Jarden call a ``prefect Frobenius field'';
this is indeed the case for any model of $T_{\gal,k}$.

\section{Proof of Theorem~\ref{thm:main}}
\label{sect:proof}

In this section we prove Theorem~\ref{thm:main}
without using Theorem~\ref{thm:reduce}: we construct
the map $\mmu\colon \KT \to \KmotQ$, check its properties, and
prove uniqueness in the case $\gal = \Zhut$.
For the whole section, we fix a
torsion-free pro-cyclic group $\gal$
and a field $k$ of characteristic zero. We also fix
the theory $T := T_{\gal,k}$ we will be working in.

\subsection{Some preliminary lemmas}
\label{subsect:lem}

We will need the following basic property of the generalized Artin symbol.

\begin{lem}\label{lem:bij}
Suppose we have the following commutative diagram of varieties over $k$,
where the maps $f_1\colon V \to W_1$ and
$f_2\colon V \to W_2$ are Galois covers with groups $G_1$ and $G_2$,
respectively.
We have naturally $G_1 \subset G_2$.
\[
\begin{array}{c@{\qqquad}c}
\Rnode{V}{V}\\[4ex]
\Rnode{W1}{W_1} & \Rnode{W2}{W_2}
\end{array}
\ncline{->>}{V}{W1}\Bput{f_1}
\ncline{->>}{V}{W2}\Aput{f_2}
\ncline{->>}{W1}{W2}\Bput{\phi}
\]
Suppose additionally that
$C_1$ is a conjugacy class of subgroups of $G_1$ and
$C_2 := C_1^{G_2}$ is the induced conjugacy class of subgroups of $G_2$.
Then for any field $K \supset k$,
the image under $\phi$ of $X_1(K) := \me{w_1 \in W_1(K) \mid \Ar(w_1) = C_1}$
is $X_2(K) := \me{w_2 \in W_2(K) \mid \Ar(w_2) = C_2}$.
Moreover, the size of the fibers of the induced map
$X_1(K) \to X_2(K)$
is $\frac{|G_2|\cdot|C_1|}{|C_2|\cdot|G_1|}$.
\end{lem}

\BLAU{
\begin{proof}
Fix a field $K$.

``$\phi(X_1)\subset X_2$'':
Suppose $w_1 \in X_1(K)$ and
$v \in f_{1}^{-1}(w_{1})$ is a preimage in $V(\aclK)$.
Then $\Fr(v) \in C_1$. But $v$ is also a preimage of $\phi(w_1)$,
so $\Ar(\phi(w_1))$ contains $\Fr(v)$ and is therefore equal to $C_1^{G_2}$.

``$\phi(X_1)\supset X_2$'':
Suppose $w_2 \in X_2(K)$. As $\Ar(w_2)$ contains
$C_1$, we may choose a preimage $v \in V(\aclK)$ of $w_2$ with $\Fr(v) \in C_1$.
In particular, $\Fr(v) \subset G_1$, which means that $\GalK$
fixes $v.G_1 \in (V/G_1)(K) \cong W_1(K)$. So the image $f_1(v)$
lies in $W_1(K)$. As $\Ar(f_1(v))$ contains
$\Fr(v)$, we have $\Ar(f_1(v)) = C_1$, so $f_{1}(v)$ is a preimage of $w_{2}$
lying in $X_1(K)$.

To compute the fiber size, fix a group $Q \in C_1$.
For any $w_2 \in X_2(K)$, consider the set
$F_2 := \me{v\in f_2^{-1}(w_2) \mid \Fr(v) = Q}$.
By Lemma~\ref{lem:13}, this set has $\frac{|G_2|}{|C_2|}$ elements.

For any $v \in F_2$,
$f_1(v)$ lies in $X_1(K) \cap \phi^{-1}(w_2)$
(lying in $X_1(K)$ holds true by the same argument as in ``$\phi(X_1)\supset X_2$'').
So the sets $F_1(w_1) := \me{v\in f_1^{-1}(w_1) \mid \Fr(v) = Q}$,
$w_1 \in X_1(K) \cap \phi^{-1}(w_2)$,
form a partition of $F_2$.
By Lemma \ref{lem:13}, each set $F_1(w_1)$ has $\frac{|G_1|}{|C_1|}$ elements,
so $w_2$ has $\frac{|G_2|\cdot|C_1|}{|C_2|\cdot|G_1|}$ preimages
in $X_1(K)$.
\end{proof}
}

The following lemma can be seen as a qualitative version of Chebotarev's
density theorem, where the finite fields have been replaced by models of
our theory. However, the proof is much easier than the one of the usual
density theorem.

\begin{lem}\label{lem:chebo}
Suppose $(\VGW, C)$ is a colored Galois cover with $C \ne \emptyset$.
\begin{enumerate}
\item
There exists a model $K\models T$
such that $X(\VGW, C)(K)$ contains an element which is generic over $k$.
\item
If $K\models T$ is a model such that $W$ is irreducible over $K$
and $X(\VGW, C)(K)$ is not empty, then $X(\VGW, C)(K)$ is already dense
in $W(K)$.
\end{enumerate}
\end{lem}

Part (1) follows from Theorem~23.1.1 of \cite{FJ}; part (2) follows from
Proposition~24.1.4 of \cite{FJ}. For details, see Corollary~2.9 and Lemma~2.10
of \cite{Nic:mot}: the proofs there (which are for pseudo-finite fields)
directly generalize to models of $T$.

\subsection{Existence of $\mmu$}
\label{subsect:proof}

The proof of the existence of the map $\mmu$ of Theorem~\ref{thm:main}
consists of three parts:
\begin{enumerate}
\item
Define a virtual motive associated to a colored Galois cover.
\item
Generalize this definition to Galois stratifications
and verify that the virtual motive defined in this way
only depends on the set defined by the stratification.

Using the
quantifier elimination result Lemma~\ref{lem:qe}, we thus get
a map $\mmu$ from the definable sets to the virtual motives.
\item
Check that this map $\mmu$ has all the required properties: that it
is invariant under definable bijections and compatible with
disjoint union and products (so it defines a ring homomorphism
$\KT \to \KmotQ$) and that it satisfies condition $(*)$
of Theorem~\ref{thm:main}.
\end{enumerate}

\medskip

(1)
To associate a virtual motive to a colored Galois cover
$(\VGW, C)$, one first associates a
central function $\alpha_{C}\colon G \to \bbQ$ to the coloring, and
then one uses a result from \cite{BRNA} to turn this into
a virtual motive.

More precisely,
let $C(G, \bbQ)$ be the $\bbQ$-vector space of $\bbQ$-central functions,
i.e.\ the space of functions $\alpha\colon G \to \bbQ$ such that
$\alpha(g) = \alpha(g')$ whenever $g, g' \in G$ generate
conjugate subgroups of $G$.
The following result essentially follows from Theorem~6.1
of \cite{BRNA}; see \cite{DL} or \cite{Nic:mot} for more details.

\begin{lem}\label{lem:mmu}
There exists a (unique) map $\mmu$ which associates to
each finite group $G$, each $G$-variety $V$ and each $\bbQ$-central
function $\alpha \in C(G, \bbQ)$
a virtual motive $\mmu(G\acts V, \alpha) \in \KmotQ$ and which has
the following properties:
\begin{enumerate}
\item
For any fixed $G$ and $\alpha$, the induced map
from the Grothendieck ring of $G$-varieties to
$\KmotQ$ is a group homomorphism.
\item
For any fixed $G$ and $V$, the induced map
$C(G, \bbQ) \to \KmotQ$ is $\bbQ$-linear.
\item\label{it:mmu:reg}
If $\alphareg$ is the character of the regular representation of $G$,
then $\mmu(G\acts V, \alphareg) = \mmu(V)$.
\item\label{it:mmu:pi}
Suppose $G$ is a group acting on a variety $V$, $H$ is a normal subgroup,
$\pi\colon G \surject G/H$ is the projection, and $\alpha \in C(G/H, \bbQ)$
is a $\bbQ$-central function. Then
$$\mmu(G/H\acts V/H, \alpha) = \mmu(G\acts V, \alpha\circ \pi)\pu$$
\item\label{it:mmu:ind}
Suppose $G$ is a group acting on a variety $V$, $H \subset G$
is any subgroup, and $\alpha \in C(H, \bbQ)$
is a $\bbQ$-central function. Then
$$\mmu(G\acts V, \Ind_{H}^{G}\alpha) = \mmu(H\acts V, \alpha)\pu$$
\end{enumerate}
(Several other properties are omitted. See e.g. \cite{DL}, Theorem~3.1.1
and Proposition~3.1.2 or \cite{Nic:mot}, Section~7.)
\end{lem}

Using this, one defines
$$
\mmu(\VGW, C) := \mmu(G\acts V, \alpha_{C})
\ko
$$
where $\alpha_C$ still has to be defined.

In the case of pseudo-finite fields,
one defines $\alpha_C$ to be $1$ on the set
$\me{g\in G\mid \gen{g} \in C}$ and $0$ elsewhere.
Just copying this definition does not work
when the Galois group is not $\Zhut$. The reason is that
the meaning of ``$Q \in C$'' is different when the Galois group of the
field is not $\Zhut$.
For example, ``$\me{1} \in C$'' means ``just a little part of $W$''
when $\gal = \Zhut$, whereas when $\gal$ is trivial,
it means ``the whole of $W$''.

To get a working definition for $\alpha_C$ in the non-$\Zhut$-case, one has
to recall that the Artin symbol is the image of a certain map
$\rho\colon \gal \to G$. Then one views
$\gal$ as a subgroup of $\Zhut$ and considers extensions of $\rho$ to $\Zhut$,
as described in the remark after Definition~\ref{defn:Ppart}.
In this way one naturally gets the following definition, which will turn
out to work:
\[
\alpha_{C}(g) := \begin{cases}
1 & \text{if } \Ppart(\gen g) \in C\\
0 & \text{otherwise}\pu
\end{cases}
\]

\medskip

(2)
We generalize the map $\mmu$ from colored Galois covers to Galois
stratifications in the obvious way:
$$
\mmu(\calA) := \sum_{i\in I} \mmu(\VGW[_i],C_i)
\pu
$$

Now suppose that two Galois stratifications $\calA$ and $\calA'$ define the
same set. To
check that the associated motives $\mmu(\calA)$ and $\mmu(\calA')$
are the same, we use Lemma~\ref{lem:refine}. It is enough to show that
(a) refining a stratification does not change the motive and that
(b) if two colorings of a Galois cover define the same set, then these
colorings are equal. Refinement of stratifications decomposes into
two parts: (a1) refining the underlying sets $W_i$ and (a2) refining
the Galois covers.

(a1) is straight forward.

(a2) We have to show that $\mmu(\VGW, C) = \mmu(V'\overto{G'}W, C')$
where $(\VGW, C)$ is a colored Galois cover and $(V'\overto{G'}W, C')$ is a
refinement.
By Lemma~\ref{lem:mmu} (\ref{it:mmu:pi}), it is enough to check that
$\alpha_{C'} = \alpha_{C} \circ \pi$, where $\pi\colon G'\surject G$
is the canonical map. But indeed we have, for any $g' \in G'$:
\[
\begin{split}
\alpha_{C'}(g') = 1
&\iff
\Ppart(\gen{g'}) \in C'
\iff
\pi(\Ppart(\gen{g'})) \in C\\
&\iff
\Ppart(\gen{\pi(g')}) \in C
\iff
\alpha_{C}(\pi(g')) = 1
\pu
\end{split}
\]

(b) follows from Lemma~\ref{lem:chebo}.
Suppose that $C_1$ and $C_2$ are two different colorings of the Galois cover
$\VGW$. Then there exists a conjugacy class $C \subset C_1 \ohne C_2$
(or vice versa), and
the lemma yields a model $K$ such that $X(\VGW, C_1)(K) \supsetneq X(\VGW,
C_2)(K)$.

\medskip

(3)
Checking that $\mmu$ is compatible with disjoint unions
and with products is straight forward. (For the products, one uses
a product property of the map $\mmu$ of Lemma~\ref{lem:mmu};
cf.\ Lemma~8.7 of \cite{Nic:mot}).

We have to check condition $(*)$ of Theorem~\ref{thm:main}, i.e.
\[
(*)\qquad \mmu(X(\VGW, \me{1})) = \frac{1}{|G|}\mmu(V)\ko
\]
where all prime factors of $|G|$ lie in $P$.
By Lemma~\ref{lem:mmu}, (\ref{it:mmu:reg}), it is enough to show
that $\alpha_{\me{1}} = \frac{1}{|G|}\alphareg$, where
$\alphareg$ is the character of the regular representation of $G$.
But indeed:
$\alpha_{\me{1}}(g) = 1$ if $\Ppart(\gen{g}) = \me{1}$ and
$\alpha_{\me{1}}(g) = 0$ otherwise. As
all prime factors of $|G|$ lie
in $P$, we have $\Ppart(\gen{g}) = \gen{g}$, so 
$\alpha_{\me{1}}(g) = 1$ only if $g = 1$.

The last property to prove is invariance under definable bijections.
We do this by first reducing the problem several times, until we are in
the situation of Lemma~\ref{lem:bij}.
\begin{itemize}
\item
A definable bijection $\phi\colon X \to X'$ also yields bijections
to the graph of $\phi$, so we may suppose that the map $X \to X'$
is the restriction of a projection (which we also denote by $\phi$).
\item
We may suppose $X = X(\VGW, C)$ by treating each component of $X$
separately. (Replace $X'$ by the image of that component.)
\item
Next we may suppose $X' = X(\VGW['], C')$ by treating each component of $X'$
separately. One easily checks that the new preimage $X$
is still defined by a single Galois cover.
(Note that for this, the order of this and the previous
step is important.)

\BLAU
To do this, we intersect the original set $X'$ with $W'$,
so $X$ is intersected with $\phi^{-1}(W')$; therefore, $X$ is still defined
by a single Galois cover.

\item
Using the density statements of Lemma~\ref{lem:chebo} and Noetherian
induction,
we may suppose that the map $\phi\colon W \to W'$ is finite and \'etale.
By refining the Galois covers, we may suppose $V = V'$.
\end{itemize}

We now have the following diagram:
\[
\begin{array}{c@{\qqquad}c}
\Rnode{V}{V}\\[4ex]
\Rnode{W1}{W} & \Rnode{W2}{W'}
\end{array}
\ncline{->>}{V}{W1}\Bput{G}
\ncline{->>}{V}{W2}\Aput{G'}
\ncline{->>}{W1}{W2}\Bput{\phi}
\]
By decomposing once more and using Lemma~\ref{lem:mmu},
we may suppose that $C$ consists of a single
conjugacy class of subgroups of $G$ and $C' = C^{G'}$ is the induced
class in $G'$.
Moreover, we get $\frac{|C|}{|G|} = \frac{|C'|}{|G'|}$, as
by assumption $\phi$ induces a bijection
$X(\VGW, C)(K) \to X(V\overto{G'}W', C')(K)$.
(Choose $K$ using
Lemma~\ref{lem:chebo} such that $X(V\overto{G'}W', C')(K)$ is not empty.)

We want to show $\mmu(\VGW, C) = \mmu(V\overto{G'}W', C')$.
By Lemma~\ref{lem:mmu} (\ref{it:mmu:ind}),
it is enough to show that $\alpha_{C'} = \Ind_G^{G'}\alpha_{C}$.

Set
\[
\begin{aligned}
\Chut &:= \me{\gen{g} \subset G \mid \alpha_C(g) = 1}
 = \me{\gen{g} \subset G \mid \Ppart(\gen{g}) \in C} \quad\text{and}
\\
\Chut' &:= \me{\gen{g'} \subset G' \mid \alpha_{C'}(g') = 1}
 = \me{\gen{g'} \subset G \mid \Ppart(\gen{g'}) \in C'}
\pu
\end{aligned}
\]
We want to understand the relation between $\Chut$ and $\Chut'$.
For this,
consider the map $\eta\colon \Chut' \to C', Q \mapsto \Ppart(Q)$.
It maps $\Chut$ to $C$. We claim that $\Chut$ is exactly the preimage
of $C$ under $\eta$. For this, we have to verify that for any
group $Q \in \Chut'$ with $\Ppart(Q) \in C$,
we already have $Q \subset G$. Indeed:
$Q$ is abelian, so it is contained in $\Norm_{G'}(\Ppart(Q))$,
and $\Norm_{G'}(\Ppart(Q))$ is contained in $G$.

Now using that $C$ consists of a single conjugacy class and that
$\eta$ commutes with conjugation,
we arrive at two conclusions: $\Chut' = \Chut^{G'}$ and
$\frac{|\Chut'|}{|C'|} = \text{fiber size of $\eta$} = \frac{|\Chut|}{|C|}$.

Using this, we can finally compute $\Ind_G^{G'}\alpha_{C}$.
For any $g' \in G'$, we have
\[
\Ind_G^{G'}\alpha_{C}(g') = \frac{1}{|G|}
\#\me{h \in G' \mid \gen{hg'h^{-1}} \in \Chut}
\pu
\]
This is zero if $\gen{g'} \notin \Chut^{G'} = \Chut'$. Otherwise:
\[
\dots = \frac{1}{|G|}\cdot |\Chut|\cdot |\Norm_{G'}(\gen{g'})|
= \frac{|\Chut|}{|G|}\cdot \frac{|G'|}{|\Chut'|} = 1
\pu
\]
(In the last equality, we combine $\frac{|C|}{|G|} = \frac{|C'|}{|G'|}$ and
$\frac{|\Chut'|}{|C'|} = \frac{|\Chut|}{|C|}$.)

\subsection{The uniqueness statement}
\label{subsect:unique:proof}

We now prove the uniqueness of the map $\mmu$ in
the case of pseudo-finite fields. For this, we only
need following properties of $\mmu$:
it extends the usual map $\mmu\colon\Kvar \to \Kmot$,
it is invariant under definable bijections,
it is compatible with disjoint unions, and for 
any Galois cover $\VGW$, the equality
\[
(*)\qquad \mmu(X(\VGW, \me{1})) = \frac{1}{|G|}\mmu(V)
\]
holds.

In particular, we will not need that $\mmu$ is compatible
with products.

\begin{proof}[Proof of uniqueness in Theorem~\ref{thm:main}]
By Lemma~\ref{lem:qe} (quantifier elimination to Galois formulas)
and compatibility with disjoint unions, it is enough
to prove uniqueness for definable sets of the form
$X(\VGW, C)$, where $(\VGW, C)$ is a colored Galois cover
and $C = Q^G$ consists of a single conjugacy class of cyclic subgroups of $G$.

We proceed by induction on $|G|$ and $|Q|$. (We will suppose that
the statement is true for $G$ of the same size and $Q$ smaller
and vice versa.)

\BLAU Kein Ind. Anfang noetig.

Suppose first that $Q$ is not normal in $G$.
Let $G' := \Norm_G(Q)$ be its normalizer and $W' := V/G'$.
Note that $C' := Q^{G'} = \me{Q}$.
By induction, we know $\mmu(X(V\overto{G'}W', C'))$.
We have $\frac{|G'|}{|C'|} = \frac{|G|}{|C|}$, so Lemma~\ref{lem:bij}
implies that the map $W' \to W$ induces a bijection
$X(V\overto{G'}W', C') \to X(\VGW, C)$. So by
assumption $\mmu(X(\VGW, C)) = \mmu(X(V\overto{G'}W', C'))$.

Now suppose $Q$ is normal in $G$ (and in particular $C = \me{Q}$).
Let $G' := G/Q$
and $V' := V/Q$. We know
$\mmu(X(V'\overto{G'}W, \me{1}))$ by $(*)$, and
we have $X(V'\overto{G'}W, \me{1}) = X(\VGW, C_1)$, where
$C_1 = \{Q_1 \in \Psub(G) \mid Q_1\subset Q\}$ consists of all
(cyclic) subgroups of $G$ contained in $Q$.
But for any strict subgroup $Q_1 \subsetneq Q$, we know
$\mmu(X(\VGW, Q_1^G))$ by induction. So
$\mmu(X(\VGW, Q))$ is the only (up to now) unknown term
in the equation
\[
\mmu(X(\VGW, C_1)) = \!\!\!\!\!\!\!\!\!\!\!\!
\sum_{\genfrac{}{}{0pt}1{C_2 \subset C_1}{C_2 \text{ one conjugacy class}}}
\!\!\!\!\!\!\!\!\!\!\!\! \mmu(X(\VGW, C_2))
\pu
\]
\end{proof}

\section{Maps between Grothendieck rings}
\label{sect:reduce}

In this section we first prove the existence of the map $\theta_\iota$ between the different
Grothendieck rings $\KT$ (Theorem~\ref{thm:reduce}) and then
apply this to get Proposition~\ref{prop:example}.
Finally we check a compatibility between the maps $\theta_\iota$ and the maps $\mmu$.

\subsection{Existence of the maps $\theta_{\iota}$}

Recall the statement of the theorem.
We have a field $k$ and an inclusion of torsion-free pro-cyclic groups
$\iota\colon \gal_2 \inject \gal_1$.
For simplicity, we will now identify $\gal_2$ with $\iota(\gal_2) \subset \gal_1$.
Denote by
$T_i := T_{\gal_i, k}$ the theory of perfect PAC fields with
Galois group $\gal_i$ and which contain $k$.

The map $\theta := \theta_\iota\colon \KT[2] \to \KT[1]$ was defined
as follows.
Any model $K_1$ of $T_1$ yields a model
$K_2 := \acl{K}_1^{\gal_2}$ of $T_2$.
For any $X_2 \subset \bbA^n$ definable in $T_2$, we defined
$\theta(X_2)(K_1) = X_2(K_2) \cap K_1^n$.

What we have to check is:
\begin{enumerate}
\item
$X_2(K_2) \cap K_1^n$ is definable (uniformly for all $K_1$).
\item
If there is a definable bijection $X_2 \to X_2'$ in $T_2$, then
there is also a definable bijection $\theta(X_2) \to \theta(X_2')$ in
$T_1$.
\item
$\theta$ is a ring homomorphism, i.e.\ compatible with
disjoint unions and products.
\end{enumerate}

The third statement is clear by definition.

(1)
Any definable set $X_2$ of $T_2$ can be written as disjoint union
of sets of the form $X(f\colon\VGW, C_2)$, where $C_2$ is a conjugacy class
of permitted subgroups of $G$, so it is enough to prove that $\theta$
maps such sets to definable ones. We claim:
$\theta(X(\VGW, C_2)) = X(\VGW, C_1)$, where $C_1$ is defined as follows:
Let $M$ be the set of
homomorphisms $\rho_1 \colon \gal_1 \to G$ such that
$\rho_1(\gal_2) \in C_2$. Then $C_1$ is the set of images of these
homomorphisms $M$. In a formula:
\[
C_1 = \me{\im \rho_1 \mid \rho_1 \colon \Gal_1 \to G, \rho_1(\Gal_2) \in C_2}
\pu
\]
We have to check: For any model $K_1$ of $T_1$ and any element
$w \in W(K_1)$, we have $w \in  X(\VGW, C_1)(K_1)$ if and only if
$w \in X(\VGW, C_2)(K_2)$, where $K_2 = \acl{K}_1^{\gal_2}$
as above.

Choose an element $v \in V(\acl{K_1})$ with $f(v) = w$.
We get a homomorphism $\rho_1\colon \gal_1 \to G$ defined by
$\sigma(v) = v.\rho_1(\sigma)$ for any $\sigma \in \gal_1$.
Of course the restriction $\rho_2 := \rho_1\auf{\gal_2}$ satisfies the
same property.
By definition, we have $w \in  X(\VGW, C_1)(K_1)$
if and only if $\im \rho_1 \in C_1$ and $w \in  X(\VGW, C_2)(K_2)$
if and only if $\im \rho_2 = \rho_1(\Gal_2) \in C_2$. So we have to check that
for any $\rho_1\colon \gal_1 \to G$ we have
$\im\rho_1 \in C_1$ if and only if $\rho_1(\Gal_2) \in C_2$.

``$\Leftarrow$'' is clear by the definition of $C_1$.

``$\Rightarrow$'': Suppose $Q_1 := \im \rho_1 \in C_1$. By the definition of
$C_1$, there is a homomorphism $\rho_1' \in M$ with
$\im\rho_1' = Q_1$. As $\gal_1$ is pro-cyclic, homomorphisms
$\gal_1 \to Q_1$ are determined by the image of a generator, so we can write
$\rho_1 = \alpha \circ \rho_1'$ for some automorphism
$\alpha \in \Aut(Q_1)$. As $Q_1$ is cyclic, all its subgroups are
characteristic subgroups, so
$\rho_1(\gal_2) = \alpha(\rho_1'(\gal_2)) = \rho_1'(\gal_2)
\in C_2$. This implies $\rho_{1} \in C_{1}$.

\BLAU Dieses Argument geht schief z.B. bei $Gal_2 = \Zhut \subset \Zhut*\Zhut =
\gen{a_1,a_{-1}} = \gal_1$: Die Abbildungen $\rho_1,\rho_{-1}$ nach $\bbZ/2$
mit $\rho_i(a_i) = 1$, $\rho_i(a_{-i}) = 0$ haben beide das selbe Bild, aber
eingeschraenkt auf $\bbZ$ nicht mehr. Man sieht, dass die Menge $M$,
die von $C_1$ definiert wird, zu gross ist.

(2) Suppose $X_2 \subset \bbA^n$ and $X_2' \subset \bbA^{n'}$
are two definable sets in $T_2$ and
$f\colon X_2 \to X_2'$ is a definable bijection. We have to show
that there is a $T_1$-definable bijection $\theta(X_2) \to \theta(X_2')$.
Indeed, we will check that $\theta(f)$ is such a bijection.
In other words we have to verify the following statement:

Let $K_1$ be any model of $T_1$ and $K_2 = \acl{K}_1^{\gal_2}$.
Then for any $x \in X_2(K_2)$ and $x' := f(x) \in X_2'(K_2)$,
we have $x \in K_1^n$ if and only if $x' \in K_1^{n'}$.

Suppose $x \notin K_1^n$.
Then there exists a $\sigma \in \Gal(K_2/K_1)$ moving $x$. But
$\sigma(X_2(K_2)) = X_2(K_2)$, so $\sigma(x) \in X_2$. As $f$ is
injective on
$X_2(K_2)$, this implies $\sigma(f(x)) = f(\sigma(x)) \ne f(x)$, so
$f(x) \notin K_1^{n'}$.

The other direction works analogously.
\hspace*{\fill}\qed

\subsection{$\mmu$ is not injective}
\label{subsect:reduce:ex}

As an example application of the maps $\theta_\iota$,
we will now prove Proposition~\ref{prop:example}. To this end,
we will construct a pair of definable sets $X_1$ and
$X_2$ such that $\mmu(X_1) = \mmu(X_2)$ but 
$\mmu(\theta_\iota(X_1)) \ne \mmu(\theta_\iota(X_2))$
for a suitable map $\iota\colon \gal \inject \gal$.
(In fact, we will construct a whole bunch of such pairs.)

\begin{proof}[Proof of Proposition~\ref{prop:example}]
Recall that $\gal$ is a non-trivial subgroup of $\Zhut$,
i.e.\ $\gal = \prod_{p \in P} \bbZ_p$, where $P$ is a non-empty
set of primes.

For $n \in \bbN_{\ge1}$,
consider the group homomorphism
$\iota\colon \gal \inject \gal, \sigma \mapsto \sigma^n$.
Applying Theorem~\ref{thm:reduce} to this map gives an endomorphism
$\theta_n$ of $\KT[\Zhut,k]$, which can be explicitly computed
on sets defined by Galois covers as follows. Let $(\VGW, C_2)$ be a colored Galois
cover. The computation in the proof of Theorem~\ref{thm:reduce} shows that
$\theta_n(X(\VGW, C_2)) = X(\VGW, C_1)$, where
$C_1 = \me{Q \in \Psub(G) \mid Q^n \in C_2}$ consists of those
permitted subgroups of $G$ whose subgroups of $n$-th powers lie in $C_2$.

Note that $\theta_n$ is interesting only if $n$ has prime factors which lie in $P$;
otherwise, $n$ and $|Q|$ are coprime for any permitted subgroup $Q \subset G$,
which implies $Q = Q^n$, $C_1 = C_2$, and $\theta_n = \id$.

Now let $\VGW$ be any non-trivial Galois cover
such that all prime factors of $|G|$
lie in $P$, and define $X := X(\VGW, \me{\id})$.
By condition $(*)$ of Theorem~\ref{thm:main},
we have $\mmu(X) = \frac{1}{|G|}\mmu(V)$,
so $\mmu(X \times G) = \mmu(V)$.
(Here $G$ is interpreted as a discrete set.)
However, we will see that for $n = |G|$, we have
$\mmu(\theta_n(X \times G)) \ne \mmu(\theta_n(V))$.

As $\theta_n$ is
the identity on $\Kvar$, we have $\theta_n(V) = [V]$. On the other hand,
the subgroup of $n$-th powers of any cyclic subgroup of $G$ is trivial,
so $\theta_n(X) = [X(\VGW, \Psub(G))] = [W]$
and $\theta_n(X\times G) = [W \times G]$. But
$V$ and $W \times G$ are two varieties with a different number
of irreducible components of maximal dimension, so 
$\mmu(\theta_n(X \times G)) \ne \mmu(\theta_n(V))$.
\end{proof}

\BLAU
Beispiel: Cover $x \mapsto x^2$ auf $\bbA^1\ohne\me{0}$.
Dann sagt das obige Gegenbeispiel, dass
es keine defbare Bijektion gibt zwischen
$\me{\text{quadrate}}\dcup \me{\text{quadrate}}$ und $\bbA^1\ohne\me{0}$.
Daraus folgt: Es gibt keine defbare Bijektion zwischen 
$\me{\text{quadrate}}$ und $\me{\text{nicht-quadrate}}$. In jedem festen
PSF-Koerper gibt's zwar so was, aber offenbar nicht uniform fuer alle
PSF-Koerper.

\subsection{Compatibility of $\mmu$ and $\theta_\iota$}
\label{subsect:coincide}

We prove the following compatibility statement:

\begin{prop}\label{prop:coincide}
Suppose $k$ is a field of characteristic zero
and $\gal_2 \subset \gal_1$ are two torsion-free pro-cyclic groups
such that $\gal_1/\gal_2$ is torsion-free, too.
We use the following notation: $T_i := T_{\gal_i, k}$
(for $i=1,2$) are the corresponding theories,
$\mmu^i\colon \KT[i] \to \KmotQ$ are the maps of Theorem~\ref{thm:main},
and $\theta\colon\KT[2] \to \KT[1]$ is the map provided by
Theorem~\ref{thm:reduce} applied to the inclusion $\gal_2 \subset \gal_1$.
Then we have:
\[
\mmu^2 = \mmu^1 \circ \theta
\pu
\]
\end{prop}

\begin{proof}
For $i=1, 2$ let $P_{i}$ be the set of primes such that
$\gal_i = \prod_{p \in P_i} \bbZ_p$. We have
$P_2 \subset P_1$, and
$\gal_2$ is just the factor of $\gal_1$ corresponding to $P_2$.
We will write $\Psub_i$ resp.\ $\Ppart_i$
for the permitted subgroups and the permitted part
to distinguish between the two Galois groups.

We only have to verify the statement for sets of the form
$X(\VGW, C_2)$, where $(\VGW, C_2)$ is a colored Galois cover
for $T_2$.
By the proof of Theorem~\ref{thm:reduce}, we have
$\theta(X(\VGW, C_2)) = X(\VGW, C_1)$, where
$C_1$ consists of the images of those maps $\rho\colon \gal_1 \to G$
which satisfy $\rho(\gal_2) \in C_2$.
As $\gal_2$ is a direct factor of $\gal_1$,
we get $C_1 = \me{Q \in \Psub_1(G) \mid \Ppart_2(Q) \in C_2}$.

Now recall the definition of $\mmu^i$:
$\mmu^i(X(\VGW, C_i)) = \mmu(G\acts V, \alpha_{C_i})$, where
\[
\alpha_{C_i}(g) := \begin{cases}
1 & \text{if } \Ppart_i(\gen g) \in C_i\\
0 & \text{otherwise.}
\end{cases}
\]
But $\Ppart_1(\gen g) \in C_1$ if and only if
$\Ppart_2(\Ppart_1(\gen g)) = \Ppart_2(\gen g) \in C_2$,
so $\alpha_{C_1} = \alpha_{C_2}$, and the claim is proven.
\end{proof}

\section{Open problems}
\label{sect:open}

\subsection{Uniqueness of $\mmu$}

In the case of pseudo-finite fields, the conditions given in
Theorem~\ref{thm:main} are enough to render $\mmu$ unique.
One would like to have a similar uniqueness statement in the other cases.
Unfortunately, the condition
\[(*)\qquad\mmu(X(\VGW, \me{1})) = \frac{1}{|G|}\mmu(V)
\]
is false in general if $|G|$ has prime factors not in $P$
(where $\gal = \prod_{p \in P} \bbZ_p$). For algebraically
closed fields for example, we have
$\mmu(X(\VGW, \me{1})) = \mmu(W)$, which is not equal to
$\frac{1}{|G|}\mmu(V)$ unless $G$ is trivial.

The first question is: is the weak version of $(*)$
(when one requires all prime factors of $|G|$ to lie in $P$)
enough to get uniqueness? And if not:
is there some other nice condition rendering $\mmu$ unique?
One fact suggesting that the weak condition might already be strong enough
is that this is true indeed for algebraically closed fields.

\BLAU Ausserdem: Die Abbildungen $\mmu \circ \theta_n$ erfuellen $(*)$
nicht; das sind also schon mal keine Gegenbeispiele.

\subsection{From motives to measure}

The parallels between the definitions of the virtual motive associated to
a definable set and the measure of such a set (\cite{CDM}, \cite{i:dm})
suggest that one should
be able to extract the measure from the motive. More precisely,
fix a perfect PAC field $K$ of characteristic zero with pro-cyclic Galois
group $\gal$.
Note that there are two theories around now:
$T_{\gal,K}$, the theory of pseudo-finite fields containing $K$
(which is not complete) and $\Th(K)$, the (complete) theory of $K$ itself.

Denote by $\dim\colon \KO(\Th(K)) \to \bbN$
the dimension of \cite{CH:dim} (which needs not coincide
with the usual dimension for varieties: only components ``visible over
$K$'' are considered) and by $\mu\colon \KO(\Th(K)) \to \bbQ$
the measure of \cite{i:dm}.
The question is whether a dotted map in the following diagram
exists making the diagram commutative.
\[
\begin{array}{c@{\qquad}c}
\Rnode{KT}{\KT[K]} & \Rnode{ThK}{\KO(\Th(K))}  \\[4ex]
\Rnode{Kmot}{\KmotQ[K]} & \Rnode{NQ}{\bbN \times \bbQ}
\end{array}
\ncline{->>}{KT}{ThK}
\ncline{->}{KT}{Kmot}\Aput{\mmu}
\ncline{->}{ThK}{NQ}\Aput{(\dim, \mu)}
\ncline[linestyle=dashed,dash=.5mm .6mm]{->}{Kmot}{NQ}
\]

If $K$ is algebraically closed, then this is obviously true:
In this case $\mu(V)$ is just the number of irreducible components
of maximal dimension
of $V$, and both this and the dimension of $V$ (which is the usual
one in this case) can be seen in the corresponding motive.

If $K$ is pseudo-finite, this is true, too: Let $X$ be a definable
set of $T_{\gal,K}$. Then it makes sense to speak about $X(F)$ for finite
fields $F$ of almost all characteristics.
Lemma~3.3.2 of \cite{DL} states that for almost all characteristics,
the number of points $|X(F)|$ is encoded in the motive.
(Not very surprisingly, it is the trace of the Frobenius automorphism
on the motive.)
The dimension and the measure of $X$ in $K$ can be computed from these
cardinalities.

The way one extracts the dimension and the measure
from the motive seems quite different in the two above cases.
This suggests that one might get interesting new insights by
generalizing this to arbitrary pro-cyclic Galois groups.

\BLAU Selbst wenn man es nicht schafft, anzugeben, wie man das
Mass aus dem Motiv ablesen kann, koennte man vielleicht auf
nicht-konstruktive Weise zeigen, dass das Mass ueber das Motiv
faktorisiert.

\subsection{Larger Galois groups for the maps $\theta_\iota$}

The quantifier elimination result of \cite{FJ} does not only work
for fields with pro-cyclic Galois groups, but for some larger Galois
groups as well. (The Galois group has to satisfy what Fried-Jarden
call the ``embedding property''.)
It seems plausible that Theorem~\ref{thm:reduce} should be
generalizable to this context as well. However the proof will need
some modifications. Indeed for $\gal_1 = \Zhut * \Zhut = \gen{a,b}$
and $\gal_2 = \gen{a} \subset \gal_1$, one can construct
a $T_{2}$-definable set $X = X(\VGW, C)$ such that $\theta(X)$
is not definable using the same Galois cover $\VGW$.

\BLAU
Suppose $\gal_1 = \Zhut * \Zhut = \gen{a,b}$, $\gal_2 = \gen{a} \subset
\gal_1$, and suppose $f\colon\VGW$ is a Galois cover with group
$G = \me{\pm1}$. Consider the following two maps:
\[
\begin{aligned}
\rho\colon \gal_1\to G&, a\mapsto 1, b\mapsto -1\\
\rho'\colon \gal_1\to G&, a\mapsto -1, b\mapsto -1
\end{aligned}
\]
and suppose there are two elements $v, v' \in V(\aclK_1)$
with $f(v), f(v')\in K_1$ and such that
$\sigma(v) = v.\rho(\sigma)$
and $\sigma(v') = v'.\rho'(\sigma)$ for $\sigma \in \gal_1$.
(Do such elements
necessarily exist? A version of Chebotarev's density theorem for
fields with bigger Galois group would be interesting.)

\BLAU
Then $v$ and $v'$ have the same Artin symbol relative to $T_1$,
namely $\{G\}$ but different Artin symbols relative to $T_2$,
namely $\{1\}$ resp. $\{G\}$. This implies that
$\theta(X(\VGW, \me{G}))$ can not be defined using the
Galois cover $\VGW$.

\subsection{Larger Galois groups for the maps $\mmu$}

Another natural question is whether the map $\mmu$ can also be defined
for fields with larger Galois group.
However, in \cite{i:dm} we already showed that the measure of \cite{CDM}
does not extend to this generality. Indeed,
no measure exists for example if the
Galois group is $\Zhut * \Zhut$. This suggests
that it is neither possible to associate motives to definable sets
of such theories. Probably, $T_{\Zhut * \Zhut, k}$ contains
too many definable bijections so that the corresponding Grothendieck
ring gets too small. One
might even hope to show that $\KT[\Zhut * \Zhut, k]$ is trivial.

\BLAU Das Problem bei Uebernahme des Beispiels aus \cite{i:dm} ist, dass
dort uniqueness explizit verwendet wurde. Man muesste also statt dessen
vermutlich ganz konkrete Galois-Cover basteln...

\subsection{What exactly do we know about $\KT$?}

We showed that the maps $\mmu$ do not yield the full information
about the definable sets and we showed how additional information can
be obtained using the maps $\theta_\iota$. The question is now:
how much information do we get using all maps $\theta_\iota$?
More precisely, suppose $X_1$ and $X_2$ are two definable sets
in $T_{\gal,k}$, and
suppose that for any injective endomorphism $\iota\colon \gal \inject \gal$
we have $\mmu(\theta_{\iota}(X_1)) = \mmu(\theta_{\iota}(X_2))$.
What does this tell us about $X_1$ and $X_2$ (as elements of $\KT$)?

The best we could hope would be $[X_1] = [X_2]$, but this is
wrong in the case $\gal = \me{1}$: There are no non-trivial maps
$\theta_{\iota}$, and the map $\mmu\colon \KT \to \Kmot$ is known
to be non-injective for algebraically closed fields.

So what one could really hope for would be that ``apart from this'', the maps
$\mmu\circ\theta_\iota$ yield all additive information about the definable sets
of $T_{\gal,k}$.
The first open problem here is to give a precise meaning to this
statement.

\BLAU
Koennte man das $\mmu$ nicht ueber $\Kvar_{\bbQ}$ faktorisieren lassen?
Aber: Was ist $\Kvar_{\bbQ}$ ueberhaupt, und ist $\Kvar \to \Kvar_{\bbQ}$
injektiv?


\def\cprime{$'$}
\def\bysame{\leavevmode ---------\thinspace}
\def\og{``}\def\fg{''}
\def\cdrandname{\&}
\providecommand\cdrnumero{no.~}
\providecommand{\cdredsname}{eds.}
\providecommand{\cdredname}{ed.}
\providecommand{\cdrchapname}{chap.}
\providecommand{\cdrmastersthesisname}{Memoir}
\providecommand{\cdrphdthesisname}{PhD Thesis}

\end{document}